\documentclass{amsart}
\usepackage{latexsym,amssymb,amsmath,amsthm,amscd,graphicx,ascmac,}
\usepackage{longtable}
\usepackage{enumerate}

\numberwithin{equation}{section}
\newtheorem{theorem}{Theorem}[section]
\newtheorem*{xthm}{Theorem}

\newtheorem{lem}[theorem]{Lemma}

\theoremstyle{definition}

\theoremstyle{remark}
\newtheorem{remark}[theorem]{Remark}
\newtheorem*{xrem}{Remark}

\newcommand{\cC}{{\mathcal C}}
\newcommand{\cD}{{\mathcal D}}

\newcommand{\kM}{{\mathfrak M}}

\newcommand{\R}{{\mathbb R}}

\newcommand{\Z}{{\mathbb Z}}

\def\al{\alpha}

\def\gm{\gamma}

\def\dl{\delta}

\def\Om{\Omega}

\def\0{\emptyset}
\def\1{{\bf 1}}
\def\6{\partial}
\def\8{\infty}

\def\lt{\left}
\def\rt{\right}

\def\wt{\widetilde}

\def\Xint#1{\mathchoice
 {\XXint\displaystyle\textstyle{#1}}%
 {\XXint\textstyle\scriptstyle{#1}}%
 {\XXint\scriptstyle\scriptscriptstyle{#1}}%
 {\XXint\scriptscriptstyle\scriptscriptstyle{#1}}%
 \!\int}
 \def\XXint#1#2#3{{\setbox0=\hbox{$#1{#2#3}{\int}$}
 \vcenter{\hbox{$#2#3$}}\kern-.5\wd0}}

\def\fint{\Xint-}

\begin{document}

\title{Fractional maximal operators with weighted Hausdorff content}

\author[H.~Saito]{Hiroki~Saito}
\address{
College of Science and Technology, Nihon University,
Narashinodai 7-24-1, Funabashi City, Chiba, 274-8501, Japan
}
\email{saitou.hiroki@nihon-u.ac.jp}

\author[H.~Tanaka]{Hitoshi~Tanaka}
\address{
Research and Support Center on Higher Education for the hearing and Visually Impaired, 
National University Corporation Tsukuba University of Technology,
Kasuga 4-12-7, Tsukuba City, Ibaraki, 305-8521 Japan
}
\email{htanaka@k.tsukuba-tech.ac.jp}
\author[T.~Watanabe]{Toshikazu~Watanabe}
\address{College of Science and Technology, Nihon University, 1-8-14 Kanda-Surugadai, Chiyoda-ku, Tokyo, 101-8308, Japan}
\curraddr{}
\email{twatana@edu.tuis.ac.jp}
\thanks{
The first author is supported by 
Grant-in-Aid for Young Scientists (B) (15K17551), 
the Japan Society for the Promotion of Science. 
The second author is supported by 
Grant-in-Aid for Scientific Research (C) (15K04918), 
the Japan Society for the Promotion of Science. 
}



\date{}

\begin{abstract}
Let $n\ge 2$ be the spatial dimension.
The purpose of this note is to obtain some weighted estimates
for the fractional maximal operator $\kM_{\al}$ 
of order $\al$, $0\le\al<n$,
on the weighted Choquet-Lorentz space $L^{p,q}(H_{w}^{d})$, 
where the weight $w$ is arbitrary and 
the underlying measure is the weighted $d$-dimensional Hausdorff content $H^{d}_{w}$, $0<d\le n$. 
Concerning a dependence of two parameters $\al$ and $d$,
we establish a general form of the Fefferman-Stein type inequalities for $\kM_{\al}$.
Our results contain the works of Adams, \cite{Ad} and of Orobitg and Verdera \cite{OV} as the special cases.
Our results also imply the Tang result \cite{Ta}, 
if we assume the weight $w$ is in the Muckenhoupt $A_{1}$-class.
\end{abstract}

\maketitle

\section{Introduction}\label{sec1}
We will denote by $\cD$ the family of all dyadic cubes 
$Q=2^{-k}(j+[0,1)^{n})$, 
$k\in\Z,\,j\in\Z^{n}$. 
For a locally integrable function $f$ on $\R^{n}$, 
we define the dyadic fractional maximal operator 
$\kM_{\al}$, $0\le\al<n$, by
\[
\kM_{\al}f(x)
:=
\sup_{Q\in\cD}
\1_{Q}(x)
\fint_{Q}|f|\,{\rm d}y
l(Q)^{\al},
\]
where $\1_{E}$ denotes the characteristic function of $E$, 
the barred integral $\fint_{Q}f\,{\rm d}x$ 
stands for the usual integral average of $f$ over $Q$, 
and $l(Q)$ denotes the side length of the cube $Q$.
When $\al=0$, 
we simply write $\kM_{0}=\kM$ 
which is the Hardy-Littlewood maximal operator.
These maximal operators are fundamental tools
to study Harmonic analysis, potential theory, 
and the theory of partial differential equations
(cf. \cite{AH,G}).

If $E\subset\R^{n}$ and $0<d\le n$, then 
the $d$-dimensional Hausdorff content 
$H^{d}$ of $E$ is defined by
\[
H^{d}(E)
:=
\inf
\sum_{j=1}^{\8}
l(Q_{j})^{d},
\]
where the infimum is taken over all coverings of $E$ 
by countable families of dyadic cubes $Q_{j}$.
In \cite{OV}, 
for the Hardy-Littlewood maximal operator $\kM$,
Orobitg and Verdera proved the strong type inequality 
\[
\int_{\R^{n}}
(\kM f)^{p}\,{\rm d}H^{d}
\le C
\int_{\R^{n}}
|f|^{p}\,{\rm d}H^{d}
\]
for $d/n<p<\8$, 
and the weak type inequality
\[
\sup_{t>0}
tH^{d}(\{x\in\R^{n}:\,\kM f(x)>t\})^{1/p}
\le C
\int_{\R^{n}}
|f|^{p}\,{\rm d}H^{d},
\quad t>0,
\]
for $p=d/n$.
Here, 
the integrals are taken in the Choquet sense, that is, 
the Choquet integral of $f\ge 0$ 
with respect to a set function $\cC$ 
is defined by
\[
\int_{\R^{n}}f\,{\rm d}\cC
:=
\int_{0}^{\8}
\cC(\{x\in\R^{n}:\,f(x)>t\})\,{\rm d}t.
\]

By weights we will always mean nonnegative and locally integrable functions on $\R^{n}$. 
Let $w$ be a weight on $\R^{n}$. 
If $E\subset\R^{n}$ and $0<d\le n$, then 
the $d$-dimensional weighted Hausdorff content 
$H^{d}_{w}$ of $E$ is defined by
\[
H^{d}_{w}(E)
:=
\inf
\sum_{j=1}^{\8}
\fint_{Q_{j}}w\,{\rm d}x
l(Q_{j})^{d},
\]
where the infimum is taken over all coverings of $E$ 
by countable families of dyadic cubes $Q_{j}$.
Notice that if we let $w\equiv 1$, then 
$H^{d}_{w}=H^{d}$.
The weighted Hausdorff content, 
which is also called the weighted Hausdorff capacity,
plays a role for functions as a tool for measuring exceptional sets
in the weighted Sobolev space 
$W^{m,1}(\Om, w)$. 
For details of the weighted Hausdorff content, see \cite{AH,T}.

We say a weight $w$ belongs to $A_{1}(\R^{n})$, 
if $\kM w(x)\le Cw(x)$ holds 
for almost every $x\in\R^{n}$. 

Tang \cite{Ta} proved the following. 

\begin{xthm}[\cite{Ta}]
Suppose that a weight $w$ belongs to $A_{1}(\R^{n})$. 
Let $0\le\al<n$ and $0<d\le n$.
If $d/n<p<d/\al$, then
\[
\int_{\R^{n}}
(\kM_{\al}f)^{p}\,{\rm d}H^{d-\al p}_{w}
\le C
\int_{\R^{n}}
|f|^{p}\,{\rm d}H^{d}_{w},
\]
where the constant $C$ is independent of $f$.
\end{xthm}

Let $\wt{\cD}$ be a family of measurable sets 
having certain dyadic structure.
For $E\subset\R^{n}$, 
let $H^{d}_{\mu}(E)$ be the modified Hausdorff content 
defined by
\[
H^{d}_{\mu}(E)
:=
\inf\sum_{j}\mu(Q_{j})^{d/n},
\]
where $\mu$ is a locally finite Borel measure on $\R^{n}$
and the infimum is taken over all countable coverings of $E$ by 
$\{Q_{j}\}\subset\wt{\cD}$. 
Let 
$M_{\wt{\cD}}^{\mu}$ 
be the maximal operator adopted to the family $\wt{\cD}$ by
\[
M_{\wt{\cD}}^{\mu}f(x)
:=
\sup_{Q\in\wt{\cD}}
\1_{Q}(x)
\fint_{Q}|f|\,{\rm d}\mu.
\]
In \cite{STW}, the authors established 
the strong and the weak estimates of 
$M_{\wt{\cD}}^{\mu}$ 
on the Choquet space $L^{p}(H^{d}_{\mu})$.

To state our main results, 
we introduce the Choquet-Lorentz space 
associated with the weighted Hausdorff content.

For $0<p\le q\le\8$, 
we define the quasinorm 
\[
\|f\|_{L^{p,q}(H^{d}_{w})}
:=
\lt(
\int_0^{\8}
(t^{p}H^{d}_{w}(\{|f|>t\}))^{q/p}
\frac{{\rm d}t}{t}
\rt)^{1/q}
\]
and we denote by 
$L^{p,q}(H^{d}_{w})$ 
the set of all functions for which the above quasinorms are finite.
When $d=n$ and $w\equiv 1$,
this is classical Lorentz space.
If $p=q$, then 
$L^{p,q}(H^{d}_{w})=L^{p}(H^{d}_{w})$. 
Further, 
$L^{p,\8}(H^{d}_{w})$ 
is the set of all functions $f$ such that
\[
\sup_{t>0}
tH^{d}_{w}(\{x\in\R^{n}:\,f(x)>t\})^{1/p}
<\8.
\]

\begin{remark}\label{rem1.1}
It is well known that 
the classical Lorentz quasinorm 
$\|f\|_{L^{p,q}(\R^{n})}$ 
does not satisfy the triangle inequality.
However, if $p>1$, then 
there exists a certain norm which is equivalent to this quasinorm;
see \cite{G} for details.
In the case of Choquet-Lorentz spaces,
one can prove this fact analogously, 
which enable us that 
the quasinorm satisfies the countable subadditivity 
up to appropriate positive multiplicative constant.
\end{remark}

In \cite{Ad}, 
Adams investigated the boundedness of fractional maximal operators on unweighted Choquet-Lorentz spaces. 

\begin{xthm}[\cite{Ad}]
Let $0<d\le n$, $0\le\al<n$, 
and $p\le q$. 

\begin{enumerate}[{\rm(i)}]
\setlength{\parskip}{0cm}
\setlength{\itemsep}{0cm}
\item 
Let $d/n<p<d/\al$ and set 
$\dl=\frac{q}{p}(d-\al p)$, then 
there is a constant $C$ independent of $f$ such that
\[
\|\kM_{\al}f\|_{L^{q,p}(H^{\dl})}
\le C
\|f\|_{L^{p}(H^{d})}.
\]
\item For $p=d/n$, 
there is a constant $C$ independent of $f$ such that
\[
\|\kM_{\al}f\|_{L^{q,\8}(H^{\dl})}
\le C
\|f\|_{L^{p}(H^{d})}
\]
with $\dl=q(n-\al)$.
\end{enumerate}
\end{xthm}

In this paper,
as an extension of this result to weighted settings,
we establish the following theorems.

\begin{theorem}\label{thm1.2}
Let $w$ be any weight on $\R^{n}$. 
Let 
$0<d\le n$, 
$0\le\al<n$, 
$0\le\gm\le\al$, and 
$\dl=\frac{q}{p}(d-(\al-\gm)p)$. 
We assume that $d/n<p\le q<n/\gm$ 
and that $p<d/\al$.
Then there exists a constant $C$ independent of $f$ and $w$ such that
\[
\|\kM_{\al}f\|_{L^{q,p}(H^{\dl}_{w})}
\le C
\|f\|_{L^{p}(H^{d}_{(\kM_{\gm q}w)^{p/q}})}.
\]
\end{theorem}

Notice that the above inequality is 
\[
\int_0^{\8}
(t^{q}H^{\dl}_{w}(\{\kM_{\al}f>t\}))^{p/q}
\frac{{\rm d}t}{t}
\le C^{p}
\int_{\R^{n}}
|f|^{p}\,{\rm d}H^{d}_{(\kM_{\gm q}w)^{p/q}}.
\]

\begin{theorem}\label{thm1.3}
Let $w$ be any weight on $\R^{n}$. 
Let 
$0<d\le n$,
$0\le\al<n$, and 
$0\le\gm\le\al$.
For $d/n\le q\le n/\al$ and $p=d/n$,
there is a constant $C$ independent of $f$ and $w$ such that
\[
\sup_{t>0}
tH^{\dl}_{w}(\kM_{\al}f>t)^{1/q}
\le C
\lt(
\int_{\R^{n}}
|f|^{p}
(\kM_{\gm q}w)^{p/q}
\,{\rm d}H^{d}
\rt)^{1/p}
\]
with $\dl=q(n-\al+\gm)$.
\end{theorem}

\begin{xrem}
In Theorem \ref{thm1.2}, 
when $\gm=0$, 
the condition $q<n/\gm$ 
should be interpreted as $q<\8$.
In both theorems \ref{thm1.2} and \ref{thm1.3}, 
if we take $\gm=0$ and $w\equiv 1$,
we obtain the Adams theorem. 
If we take $p=q$ and $\al=0$, 
then $\dl$ coincides $d$ and 
we obtain weighted estimates of 
the Orobitg and Verdera theorem.
In Theorem \ref{thm1.2}, 
if we take $p=q$ and $\gm=0$ 
and we assume $w\in A_{1}(\R^{n})$, 
then we obtain the Tang theorem. 
Finally,
let $p=q$, $\gm=\al$, and $d=n$, 
then $\dl$ becomes $n$ and 
we obtain the Fefferman-Stein inequality for the fractional maximal operator
due to Sawyer \cite{S} 
(see also \cite{CU}).
\end{xrem}

\begin{remark}\label{rem1.4}
One may expect that 
the following strong and weak type inequalities also hold:
\[
\int_0^{\8}
(t^{q}H^{\dl}_{w}(\{\kM_{\al}f>t\}))^{p/q}
\frac{{\rm d}t}{t}
\le C
\int_{\R^{n}}
|f|^{p}
(\kM_{\gm q}w)^{p/q}
\,{\rm d}H^{d},
\]
and 
\[
\sup_{t>0}
tH^{\dl}_{w}(\kM_{\al}f>t)^{1/q}
\le C
\lt(
\int_{\R^{n}}
|f|^{p}\,{\rm d}H^{d}_{(\kM_{\gm q} w)^{p/q}}
\rt)^{1/p}.
\]
However, 
since it is difficult to compare 
${\rm d}H^{d}_{w}$ and $w{\rm d}H^{d}$, 
we have no idea to show the both inequalities until now.
\end{remark}

The letter $C$ will be used for the positive finite constants that may change from one occurrence to another. 
Constants with subscripts, such as $C_1$, $C_2$, do not change in different occurrences. 

\section{Proof of Theorems}

We begin by proving strong type estimate Theorem \ref{thm1.2}.
To do this, we need the following lemma due to Lemma 1 in \cite{OV}.
Hereafter,
we always assume $p<q$.

\begin{lem}\label{lem2.1}
Let 
$0<d<n$, $0<\al<n$, $0\le\gm\le\al$, and 
$\dl=\frac{q}{p}(d-(\al-\gm)p)$.
We assume $d/n<p\le q<n/\gm$ and $p<d/\al$.
Then we have that 
\[
\|\kM_{\al}[\1_{Q}]\|_{L^{q,p}(H^{\dl}_{w})}^{p}
\le C_{n,p,d}
\fint_{Q}(\kM_{\gm q}w)^{p/q}\,{\rm d}x
l(Q)^{d}
\]
for all dyadic cubes $Q\in\cD$ 
and for any weight $w$.
\end{lem}

\begin{proof}
For fixed dyadic cube $Q$,
we let $\pi^{0}(Q)=Q$ and
$\pi^{j}(Q)$ is the smallest dyadic cube containing $\pi^{j-1}(Q)$, $j=1,2,\dots$.
Now,
we see that 
\[
\kM_{\al}[\1_{Q}](x)
=
a_0\1_{Q}(x)
+
\sum_{j=1}^{\8}
a_{j}\1_{\pi^j(Q)\setminus\pi^{j-1}(Q)}(x),
\]
where 
\[
a_{j}
=
\frac{|Q|}{|\pi^j(Q)|}
l(\pi^j(Q))^{\al}
=
l(Q)^{\al}
2^{(\al-n)j},
\quad j=0,1,\ldots.
\]
It follows that 
\begin{align*}
\|\kM_{\al}[\1_{Q}]\|_{L^{q,p}(H^{\dl}_{w})}^{p}
&=
\int_0^{\8}
(t^qH^{\dl}_{w}(\{\kM_{\al}[\1_{Q}]>t\}))^{p/q}
\,\frac{{\rm d}t}{t}
\\ &=
\sum_{j=1}^{\8}
\int_{a_{j}}^{a_{j-1}}
(t^qH^{\dl}_{w}(\{\kM_{\al}[\1_{Q}]>t\}))^{p/q}
\,\frac{{\rm d}t}{t}
\\ &\le 
\sum_{j=1}^{\8}
H^{\dl}_{w}(\pi^j(Q))^{p/q}
\int_{a_{j}}^{a_{j-1}}
t^{p-1}\,{\rm d}t
\\ &\le p^{-1}
\sum_{j=1}^{\8}
\lt(
\fint_{\pi^j(Q)}w\,{\rm d}x
l(\pi^j(Q))^{\dl}
\rt)^{p/q}
(a_{j-1})^{p}.
\end{align*}
We notice that 
\[
(a_{j-1})^{p}
=
\lt(
l(Q)^{\al}
2^{(\al-n)(j-1)}
\rt)^{p}
=
2^{(n-\al)p}
\lt(
l(Q)^{\al}
2^{(\al-n)j}
\rt)^{p}.
\]
Thus,
\begin{align*}
&\lt(
\fint_{\pi^j(Q)}w\,{\rm d}x
l(\pi^j(Q))^{\dl}
\rt)^{p/q}
(a_{j-1})^{p}
\\ 
=& 2^{(n-\al)p}
\lt(
\fint_{\pi^j(Q)}w\,{\rm d}x
l(\pi^j(Q))^{\gm q}
\rt)^{p/q}
l(\pi^{j}(Q))^{d-\al p}
l(Q)^{\al p}
2^{(\al-n)jp} \\ 
=&
2^{(n-\al)p}
\lt(
\fint_{\pi^j(Q)}w\,{\rm d}x
l(\pi^j(Q))^{\gm q}
\rt)^{p/q}
l(Q)^{d}
2^{(d-np)j} \\
\le &
C_{n,p}
\sup_{P\in\cD:\,P\supset Q}
\lt(
\fint_{P}w\,{\rm d}x
l(P)^{\gm q}
\rt)^{p/q}
l(Q)^{d}
2^{(d-np)j},
\end{align*}
where we have used the definition of $\dl$
and this implies 
\begin{align*}
\|\kM_{\al}[\1_{Q}]\|_{L^{q,p}(H^{\dl}_{w})}^{p}
&\le 
C_{n,p}
\sup_{P\in\cD:\,P\supset Q}
\lt(
\fint_{P}w\,{\rm d}x
l(P)^{\gm q}
\rt)^{p/q}
\cdot
l(Q)^{d}
\cdot
\sum_{j=1}^{\8}
2^{(d-np)j}
\\ &\le C_{n,p,d}
\fint_{Q}(\kM_{\gm q}w)^{p/q}\,{\rm d}x
l(Q)^{d},
\end{align*}
where we have used 
$d/n<p$. 
We notice that 
the condition $p<d/\al$ 
needs to be $\dl>0$.
\end{proof}

\begin{proof}[Proof of Theorem \ref{thm1.2}]
First we notice that for $0<p\le q<\8$, we have
\[
\|\kM_{\al}f\|^{p}_{L^{q,p}(H^{\dl}_{w})}
=
\frac{1}{p}
\|(\kM_{\al}f)^{p}\|_{L^{q/p,1}(H^{\dl}_{w})}.
\]
Indeed,
\begin{align*}
\|\kM_{\al}f\|^{p}_{L^{q,p}(H^{\dl}_{w})}
&=
\int_{0}^{\8}
(t^{q}H^{\dl}_{w}(\{\kM_{\al}f>t\}))^{p/q}\,\frac{{\rm d}t}{t} \\
&=
\frac{1}{p}
\int_{0}^{\8}
(t^{q/p}H^{\dl}_{w}(\{(\kM_{\al}f)^{p}>t\}))^{p/q}\,\frac{{\rm d}t}{t} \\
&=
\frac{1}{p}
\int_{0}^{\8}
(t^{q/p}H^{\dl}_{w}(\{(\kM_{\al}f)^{p}>t\}))^{1/(q/p)}\,\frac{{\rm d}t}{t} \\
&=
\frac{1}{p}
\|(\kM_{\al}f)^{p}\|_{L^{q/p,1}(H^{\dl}_{w})}.
\end{align*}
We may assume that $f\ge 0$.
For each integer $k$, let 
$\{Q_{j}^k\}_{j}$ be a family of nonoverlapping dyadic cubes $Q_{j}^k$ 
such that
\[
\{x\in\R^{n}:\,2^k<f(x)\le 2^{k+1}\}
\subset
\bigcup_{j}Q_{j}^k
\]
and
\[
\sum_{j}
\fint_{Q_{j}^k}(\kM_{\gm q}w)^{p/q}\,{\rm d}x
l(Q_{j}^k)^{d}
\le 2
H_{(\kM_{\gm q}w)^{p/q}}^{d}(\{x\in\R^{n}:\,2^k<f(x)\le 2^{k+1}\}).
\]
Set 
$g=\sum_k2^{p(k+1)}\1_{A_k}$, 
where 
$A_k=\bigcup_{j}Q_{j}^k$. 
Thus, $f^{p}\le g$.

Assume first that $1\le p$. 
Then
\[
(\kM_{\al}f)^{p}
\le
\kM_{\al p}(f^{p})
\le
\kM_{\al p}(g)
\le
\sum_{k}2^{p(k+1)}
\sum_{j}\kM_{\al p}(\1_{Q_{j}^k}).
\]
Recalling Remark \ref{rem1.1}, 
by $q/p>1$,
we have that 
\[
\frac{1}{p}
\|(\kM_{\al}f)^{p}\|_{L^{q/p,1}(H^{\dl}_{w})}
\le C
\frac{1}{p}
\sum_{k}2^{p(k+1)}
\sum_{j}
\|\kM_{\al p}(\1_{Q_{j}^k})\|_{L^{q/p,1}(H^{\dl}_{w})}.
\]
By Lemma \ref{lem2.1},
\begin{align*}
&
\frac{1}{p}
\sum_{k}2^{p(k+1)}
\sum_{j}
\|\kM_{\al p}(\1_{Q_{j}^k})\|_{L^{q/p,1}(H^{\dl}_{w})} \\
\le &
\frac{1}{p}
\sum_{k}2^{p(k+1)}
\sum_{j}
C_{n,p,d}
\fint_{Q_{j}^{k}}
(\kM_{\gm q}w)^{p/q}\,{\rm d}x
\cdot
l(Q_{j}^{k})^{d} \\
\le&
C
\sum_k2^{p(k+1)}
H^{d}_{(\kM_{\gm q}w)^{p/q}}(\{x:2^k<f(x)\le 2^{k+1}\}) \\ 
\le & 
C
\sum_k\frac{2^{2p}}{2^{p}-1}
\int_{2^{p(k-1)}}^{2^{pk}}
H^{d}_{(\kM_{\gm q}w)^{p/q}}(\{x:f(x)^{p}>t\})\,{\rm d}t \\ 
\le&
C
\int_{\R^{n}}f^{p}\,
{\rm d}H^{d}_{(\kM_{\gm q}w)^{p/q}},
\end{align*}
which proves this case. 

Assume now that $d/n<p<1$.
Since 
$f\le\sum_k2^{k+1}\1_{A_k}$, 
\[
\kM_{\al}f
\le
\sum_k2^{k+1}\sum_{j}\kM_{\al}[\1_{Q_{j}^k}].
\]
We have that, since $p<1$, 
\[
(\kM_{\al}f)^{p}
\le
\sum_k2^{p(k+1)}\sum_{j}
\kM_{\al}[\1_{Q_{j}^k}]^{p}
\]
and, hence,
\begin{align*}
\|\kM_{\al}f\|^{p}_{L^{q,p}(H^{\dl}_{w})}
&=
\frac{1}{p}
\|(\kM_{\al}f)^{p}\|_{L^{q/p,1}(H^{\dl}_{w})} \\
&\le
C
\frac{1}{p}
\sum_{k}2^{p(k+1)}\sum_{j}
\|\kM_{\al}[\1_{Q_{j}^k}]^{p}\|_{L^{q/p,1}(H^{\dl}_{w})} \\
&=
C
\sum_{k}2^{p(k+1)}\sum_{j}
\|\kM_{\al}[\1_{Q_{j}^k}]\|_{L^{q,p}(H^{\dl}_{w})}^{p} \\
&\le
C_{n,p,d}
\sum_{k}2^{p(k+1)}
\sum_{j}
\fint_{Q_{j}^{k}}
(\kM_{\gm q}w)^{p/q}\,{\rm d}x
\cdot
l(Q_{j}^{k})^{d} \\
&\le C
\int_{\R^{n}}f^{p}\,
{\rm d}H^{d}_{(\kM_{\gm q}w)^{p/q}}.
\end{align*}
This completes the proof.
\end{proof}

Next,
we prove the weak type estimate Theorem \ref{thm1.3}.
The strategy of the proof is also due to \cite{OV}.

\begin{proof}[Proof of Theorem \ref{thm1.3}]
Given $t>0$, let $\{Q_{j}\}$ be 
the family of maximal dyadic cubes $Q_{j}$ 
such that
\[
\fint_{Q_{j}}f\,{\rm d}xl(Q_{j})^{\al}
>t
\]
(we assume again, without loss of generality, that $f\ge 0$). 
Then
\[
\{x:\,\kM_{\al}f(x)>t\}
=
\bigcup_{j}Q_{j}.
\]

By Lemma 3 of \cite{OV},
\begin{equation}\label{2.1}
t^ql(Q_{j})^{\dl}
\le
\lt(l(Q_{j})^{\gm}\int_{Q_{j}}f\,dx\rt)^q
\le  C
l(Q_{j})^{\gm q}
\lt(\int_{Q_{j}}f^{p}\,{\rm d}H^{d}\rt)^{q/p}.
\end{equation}
Again, applying the proof of Lemma 2 of \cite{OV} 
to the family $\{Q_{j}\}$, 
we see the following.

There exist a subfamily 
$\{Q_{j_m}\}\subset\{Q_{j}\}$ 
and a set of non-overlapping dyadic cubes 
$\{\wt{Q}_k\}$ such that 

\begin{itemize}
\item[(i)] 
for each dyadic cube $Q$, 
\[
\sum_{Q_{j_m}\subset Q}
l(Q_{j_m})^{d}
\le 2l(Q)^{d};
\]
\item[(ii)] 
\[
\bigcup_{j}Q_{j}
\subset
\lt(\bigcup_mQ_{j_m}\rt)
\cup
\lt(\bigcup_k\wt{Q}_k\rt);
\]
\item[(iii)] 
for each $k$, 
\[
l(\wt{Q}_k)^{d}
\le
\sum_{Q_{j_m}\subset\wt{Q}_k}
l(Q_{j_m})^{d}.
\]
\end{itemize}

By (ii), we have that 
\[
t^qH^{\dl}_{w}\lt(\bigcup_{j}Q_{j}\rt)
\le
t^q
\sum_m
\fint_{Q_{j_m}}w\,{\rm d}x
l(Q_{j_m})^{\dl}
+
t^q
\sum_k
\fint_{\wt{Q}_k}w\,{\rm d}x
l(\wt{Q}_k)^{\dl}.
\]
For all $Q_{j_m}$, 
by \eqref{2.1} we see that 
\begin{align*}
t^q
\fint_{Q_{j_m}}w\,{\rm d}x
l(Q_{j_m})^{\dl}
&\le C
\fint_{Q_{j_m}}w\,{\rm d}x
l(Q_{j})^{\gm q}
\lt(
\int_{Q_{j_m}}f^{p}\,{\rm d}H^{d}
\rt)^{q/p} \\
&\le
C
\lt(
\int_{Q_{j_m}}f^{p}(\kM_{\gm q} w)^{p/q}\,{\rm d}H^{d}
\rt)^{q/p}.
\end{align*}

For $\wt{Q}_k$,
we separate the argument.

\noindent
{\bf Case $\dl\le d$.}
By (iii),
for each $k$,
\begin{align*}
l(\wt{Q}_k)^{\dl}
&=
\lt(l(\wt{Q}_k)^{d}\rt)^{\dl/d}
\le
\lt(
\sum_{Q_{j_m}\subset\wt{Q}_k}
l(Q_{j_m})^{d}
\rt)^{\dl/d}
\\ &\le
\sum_{Q_{j_m}\subset\wt{Q}_k}
l(Q_{j_m})^{\dl}.
\end{align*}
Thus, by \eqref{2.1} we have that 
\[
t^{q}
\fint_{\wt{Q}_k}w\,{\rm d}x
l(\wt{Q}_k)^{\dl}
\le C
\sum_{Q_{j_m}\subset\wt{Q}_k}
\lt(
\int_{Q_{j_m}}f^{p}(\kM_{\gm q} w)^{p/q}\,{\rm d}H^{d}
\rt)^{q/p}.
\]

\noindent
{\bf Case $\dl\ge d$.}
By (iii),
for each $k$,
\begin{align*}
t^ql(\wt{Q}_k)^{\dl}
&=
t^q(l(\wt{Q}_k)^{d})^{\dl/d} \\
&\le
\lt(
\sum_{Q_{j_m}\subset\wt{Q}_k}
(t^ql(Q_{j_m})^{\dl})^{d/\dl}
\rt)^{\dl/d}
\end{align*}
Thus, by \eqref{2.1} we have that 
\begin{align*}
t^{q}
\fint_{\wt{Q}_k}w\,{\rm d}x
l(\wt{Q}_k)^{\dl}
&\le C
\lt[
\sum_{Q_{j_m}\subset\wt{Q}_k}
\lt(
t^q
\fint_{\wt{Q}_k}w\,{\rm d}x
l(Q_{j_m})^{\dl}
\rt)^{d/\dl}
\rt]^{\dl/d}
\\
&\le C
\lt[
\sum_{Q_{j_m}\subset\wt{Q}_k}
\lt(
\int_{Q_{j_m}}f^{p}(\kM_{\gm q} w)^{p/q}\,{\rm d}H^{d}
\rt)^{qd/p\dl}
\rt]^{\dl/d}
\\
&\le C
\lt(
\sum_{Q_{j_m}\subset\wt{Q}_k}
\int_{Q_{j_m}}f^{p}(\kM_{\gm q} w)^{p/q}\,{\rm d}H^{d}
\rt)^{q/p},
\end{align*}
where we have used $qd/p\dl=n/(n-\al +\gm)>1$.
Combining altogether,
we obtain
\begin{align*}
t^q
H^{\dl}_{w}\lt(\bigcup_{j}Q_{j}\rt)
&\le C
\lt(
\sum_m
\int_{Q_{j_m}}f^{p}(\kM_{\gm q} w)^{p/q}\,{\rm d}H^{d}
\rt)^{q/p}
\\ &\le C
\lt(
\int_{\R^{n}}f^{p}(\kM_{\gm q} w)^{p/q}\,{\rm d}H^{d}
\rt)^{q/p},
\end{align*}
where the last inequality is due to 
the packing condition (i). 
This completes the proof. 
\end{proof}


\begin{thebibliography}{999}

\bibitem{Ad} D.~R.~Adams, 
\emph{Choquet integrals in potential theory},
Publ.~Mat. \textbf{42} (1998) 3--66.

\bibitem{AH} D.~R.~Adams and L.~I.~Hedberg, 
\emph{Function Spaces and Potential Theory},
Springer-Verlag, Heidelberg-New York, 1996.

\bibitem{CU} D.~Cruz-Uribe, SFO,
\emph{New proofs of two-weight norm inequalities for the maximal operator},
Georgian Math. J., \textbf{7} (2000), 33--42.

\bibitem{G} L.~Grafakos,
{\it Classical Fourier Analysis},
volume 249 of Graduate Texts in Mathematics, 
Springer, New York, 2nd edition, 2008.

\bibitem{OV} 
J.~Orobitg and J.~Verdera, 
\emph{Choquet integrals, Hausdorff content and the Hardy-Littlewood maximal operator},
Bull. London Math. Soc., \textbf{30} (1998), no.~2, 145--150.

\bibitem{STW} 
H.~Saito, H.~Tanaka and T.~Watanabe,
\emph{Abstract dyadic cubes and the dyadic maximal operator with the Hausdorff content},
Bull. Sci Math., \textbf{140} (2016), 757--773.

\bibitem{S} 
E. Sawyer,
\emph{Weighted norm inequalities for fractional maximal operators},
Proc. C.M.S. \textbf{1} (1981), 283--309.

\bibitem{Ta} L.~Tang, 
\emph{Choquet integrals, weighted Hausdorff content and maximal operators},
Georgian Math. J. \textbf{18} (2011), no.~3, 587--596

\bibitem{T} 
B.~O.~Turesson,
\emph{Nonlinear Potential Theory and Weighted Sobolev Spaces},
Lecture Notes in Mathematics, 1736. Springer-Verlag, Berlin, 2000.
\end{thebibliography}
\end{document}